\documentclass[amscd,amssymb,verbatim,11pt]{amsart}
\usepackage{epsfig}

\begin{document}

\title[Valuations]{Asymptotic valuations of sequences satisfying 
first order recurrences}

\author{Tewodros Amdeberhan}
\address{Department of Mathematics,
Tulane University, New Orleans, LA 70118}
\email{tamdeber@tulane.edu}

\author{Luis A. Medina}
\address{Department of Mathematics,
Tulane University, New Orleans, LA 70118}
\email{lmedina@math.tulane.edu}

\author{Victor H. Moll}
\address{Department of Mathematics,
Tulane University, New Orleans, LA 70118}
\email{vhm@math.tulane.edu}


\date{\today}


\begin{abstract}
Let $t_{n}$ be a sequence that 
satisfies a first order homogeneous 
recurrence $t_{n} = Q(n)t_{n-1}$, where $Q \in \mathbb{Z}[n]$. 
The asymptotic behavior of 
the $p$-adic valuation of $t_{n}$ is described under the assumption that 
all the roots of $Q$ in $\mathbb{Z}/p \mathbb{Z}$ have nonvanishing
derivative. 
\end{abstract}

\maketitle

\newcommand{\nn}{\nonumber}
\newcommand{\ba}{\begin{eqnarray}}
\newcommand{\ea}{\end{eqnarray}}
\newcommand{\no}{\noindent}
\newcommand{\realpart}{\mathop{\rm Re}\nolimits}
\newcommand{\imagpart}{\mathop{\rm Im}\nolimits}

\newtheorem{Definition}{\bf Definition}[section]
\newtheorem{Thm}[Definition]{\bf Theorem}
\newtheorem{Example}[Definition]{\bf Example}
\newtheorem{Lem}[Definition]{\bf Lemma}
\newtheorem{Note}[Definition]{\bf Note}
\newtheorem{Cor}[Definition]{\bf Corollary}
\newtheorem{Conj}[Definition]{\bf Conjecture}
\newtheorem{Prop}[Definition]{\bf Proposition}
\newtheorem{Problem}[Definition]{\bf Problem}
\numberwithin{equation}{section}

\section{Introduction} \label{sec-intro}
\setcounter{equation}{0}

The $p$-adic valuation $\nu_{p}(x)$, for $x \in \mathbb{Q}, \, x \neq 0$, is 
defined by
\begin{equation}
x = p^{\nu_{p}(x)} \frac{a}{b},
\end{equation}
\noindent 
where $a, \, b \in \mathbb{Z}$ and $p$ divides neither $a$ nor $b$. The 
value $\nu_{p}(0)$ is left undefined. 

In this paper we establish the asymptotic behavior of the $p$-adic valuation 
of sequences that satisfy first order recurrences
\begin{equation}
t_{n} = Q(n) t_{n-1}, \quad \, n \geq 1,
\label{rec-tn}
\end{equation}
\noindent
where $Q$ is a polynomial with integer coefficients. Among all the positive
integer zeros of 
$Q$, let $v$ be the maximum modulus. Take $n_{0} > v$. Then the 
recurrence (\ref{rec-tn}) is 
started at this index $n_{0}$. This ensures 
the non-vanishing of $t_{n}$. Without loss of generality, we always assume 
$n_{0} = 0$ and $t_{0} =1$.  We also adopt 
the notation $t_{n}(Q)$ while refering to the sequence 
defined by (\ref{rec-tn}).

The identity 
\begin{equation}
\nu_{p}(t_{n}(Q)) = \sum_{i=1}^{n} \nu_{p}(Q(i)),
\label{basic-sum}
\end{equation}
\noindent
shows that only the zeros of $Q$ in $\mathbb{Z}/p \mathbb{Z}$ contribute to the 
value of $\nu_{p}(t_{n}(Q))$. 
The main tool of our asymptotic analysis  will be 
Hensel's lemma. The version stated here is  reproduced from
\cite{murty3}:

\begin{Lem}[Hensel]
Let $f(x) \in \mathbb{Z}_{p}[x]$ be a polynomial with coefficients in 
the $p$-adic integers 
$\mathbb{Z}_{p}$. Write $f'(x)$ for its formal derivative. If 
$f(x) \equiv 0 \bmod p$ has a solution $a_{1}$ satisfying 
$f'(a_{1}) \not \equiv 0 \bmod p$, then there is a unique $p$-adic integer 
$a$ such that $f(a) = 0$ and  $a \equiv a_{1} \bmod p$.
\end{Lem}

We now state our main theorem. This result is an asymptotic description 
of the valuation of the sequence $t_{n}$, defined by (\ref{rec-tn}).

\begin{Thm}
\label{main}
Let $Q \in \mathbb{Z}[n]$. Assume each of the roots of $Q$ satisfies the 
hypothesis of Hensel's lemma. Let $z_{p}$ denote the number of roots of $Q$ in 
$\mathbb{Z}/p \mathbb{Z}$, that is, 
\begin{equation}
z_{p} := | \{ b \in \{ 1, \, 2, \, \ldots, p \}: Q(b) \equiv 0 \bmod p \} |. 
\label{zp}
\end{equation}
\noindent
Then the sequence $\{ t_{n} \}$, defined in (\ref{rec-tn}), obeys the estimate
\begin{equation}
\nu_{p}(t_{n}) \sim \frac{z_{p}n}{p-1} \text{ as } n \to \infty.
\end{equation}
\end{Thm}

\noindent
{\bf Motivation}. The 
most elementary example is $Q(x) = x$. Theorem \ref{main} yields 
$\nu_{p}(n!) \sim n/(p-1)$. This follows from the classical formula of 
Legendre 
\begin{equation}
\nu_{p}(n!) = \frac{n - s_{p}(n)}{p-1},
\end{equation}
\noindent
where $s_{p}(n)$ is the sum of the digits of $n$ in base $p$. 

Our motivation for Theorem \ref{main} comes from the  study of the
sequence $\{ x_{n} \}$ defined by 
\begin{equation}
x_{n} = 
\tan \sum_{k=1}^{n} \tan^{-1} k, \, \quad n \geq 1.
\end{equation}
\noindent 
This same sequence satisfies the recursive relation
\begin{equation}
x_{n} = \frac{x_{n-1} + n}{1 - nx_{n-1}}, 
\end{equation}
\noindent
with initial condition $x_{1} =1$. The first few values are 
$\{ 1, \, -3, \, 0, \, 4, \, - \frac{9}{19} \},$ and in 
\cite{bomoarctan}
it was conjectured that
$x_{n} \neq 0$ for $n \geq 4$. Later this was proved in 
\cite{tvl1} using the $2$-adic valuation of $x_{n}$. The
sequence $\{ x_{n} \}$ was  linked in \cite{tvl1} to 
\begin{equation}
\omega_{n} := (1+1^{2})(1+2^{2})(1+3^{2}) \cdots (1+n^{2}),
\end{equation}
\noindent
which can be condensated as
\begin{equation}
\omega_{n}  = (1+n^{2}) \omega_{n-1}.
\label{rec1}
\end{equation}
\noindent
This corresponds to $Q(x) = x^{2}+1$ and it 
fits into the type of recurrences considered here.

Section \ref{sec-proof} contains the proof of Theorem \ref{main} and 
Section \ref{sec-examples} presentes examples illustrating the main result. 
In the last section we propose some future directions. 

\section{The proof} \label{sec-proof}
\setcounter{equation}{0}

In the proof we assume that $Q$ has no roots in $\mathbb{N} \cup \{ 0 \}$. 
The general situation can be reduced to this one by a shift of the 
independent variable.

The conclusion of Theorem \ref{main} is trivial if $z_{p} = 0$, so we  
assume $z_{p} > 0$. 
Denote by $b_{1}, \, b_{2}, \cdots, 
b_{z_{p}}$ the  zeros of $Q$ in $\mathbb{Z}/p \mathbb{Z}$. 
The definition of $t_{n}$ yields
\begin{equation}
\nu_{p}(t_{n}) = \sum_{i=1}^{n} \nu_{p}(Q(i)).
\label{sum-0}
\end{equation}
\noindent
All sums below are assumed to run from $i=1$ to $n$. 

Only the indices congruent to $b_{j}$ modulo $p$ contribute 
to (\ref{sum-0}), thus
\begin{equation}
\nu_{p}(t_{n}) = \sum_{i \equiv b_{1} \bmod p} \nu_{p}(Q(i))
 + \cdots +  \sum_{i \equiv b_{z_{p}} \bmod p} 
\nu_{p}(Q(i))
\end{equation}
\noindent
where $1 \leq i \leq n$. For 
fixed $j \in \{ 1, \,2, \, \ldots, z_{p} \},$ we consider  the term
\begin{equation}
\sum_{i \equiv b_{j} \bmod p} \nu_{p}(Q(i)).
\end{equation}
\noindent
Hensel's lemma produces a $p$-adic integer 
\begin{equation}
\beta_{j} = \beta_{j,0} + \beta_{j,1}p + \cdots + \beta_{j,k}p^{k}  +
\label{rep1}
\cdots 
\end{equation}
\noindent
such that $\beta_{j,k} \in \{ 0, \, 1, \, \cdots, p-1 \}, \,  
\beta_{j,0} \equiv b_{j} \bmod p$ and $Q(\beta_{j}) = 0$. Observe that if the 
representation (\ref{rep1}) were finite, then $\beta_{j}$ would be 
a non-negative integer
root of $Q$. This possibility has been excluded. 
Introduce the notation
\begin{equation}
\gamma_{j,s} := \beta_{j,0} + p \beta_{j,1} + p^{2} \beta_{j,2} + 
\cdots + p^{s} \beta_{j,s}p^{s}.
\end{equation}

\medskip

\begin{Definition}
For $n \in \mathbb{N}$, let 
\begin{equation}
r_{n} = \text{Max }\{j: p^{j} \text{divides some }Q(i) \text{ for }
1 \leq i \leq n \}.
\end{equation}
\end{Definition}

\begin{Lem}
The sequence $r_{n} \to \infty$ as $n \to \infty$. Moreover, for large $n$, we
have $p^{r_{n}} \leq n^{\text{deg}(Q) + 1}$, hence $r_{n} = O( \log n)$. 
\end{Lem}
\begin{proof}
Hensel's lemma shows that $\gamma_{j,s}$ satisfies 
$Q(\gamma_{j,s}) \equiv 0 \bmod p^{s+1}$. For any given $M >0$, choose
an integer $s > M$.
Taking $n > \gamma_{j,s-1}$ we have that $i:= \gamma_{j,s-1} \in \{ 1, \, 2, 
\, \cdots, n \}$ and $p^{s} | Q(i)$. The definition of $r_{n}$ implies that
$r_{n} \geq s > M$. Therefore $r_{n} \to \infty$ as $n \to \infty$. Now observe 
that $p^{r_{n}}$ divides $|Q(i)|$ for some 
$1 \leq i \leq n$. The estimate 
\begin{equation}
p^{r_{n}} \leq |Q(i)|  \leq \text{Max}\{ |Q(1)|, \cdots, |Q(n)| \} \leq 
Cn^{\text{deg}(Q)}
\end{equation}
\noindent
gives the upper bound on $r_{n}$. The constant $C$ depends only on the 
coefficients of $Q$. 
\end{proof}

Now 
\begin{equation}
\sum_{i \equiv b_{j} \bmod p} \nu_{p}(Q(i)) = 
\sum_{i \equiv \gamma_{j,0} \bmod p}  1 + 
\sum_{i \equiv \gamma_{j,1}\bmod p^{2} }  1 + 
\cdots + 
\sum_{i \equiv \gamma_{j,r_{n}-1}  \bmod p^{r_{n}}}1,
\nonumber
\end{equation}
\noindent
where all sums range over $1 \leq i \leq n$. The bound 
\begin{equation}
\left\lfloor \frac{n}{p^{s}}  \right\rfloor \leq 
\sum_{i \equiv \gamma_{j,s} \bmod p} 1 \leq 
\left\lfloor \frac{n}{p^{s}}  \right\rfloor + 1
\label{estimate1}
\end{equation}
\noindent
yields 
\begin{eqnarray}
\sum_{i \equiv b_{j} \bmod p} \nu_{p}(Q(i)) & \geq & 
\left( \frac{n}{p} - 1 \right) + 
\left( \frac{n}{p^{2}} - 1 \right) + \cdots + 
\left( \frac{n}{p^{r_{n}}} - 1 \right)  \nonumber \\
& = & n \left( \frac{1}{p} + \frac{1}{p^{2}} + \cdots + \frac{1}{p^{n_{r}}} 
\right) - r_{n} \nonumber \\
& = & \frac{n}{p-1} \left( 1 - p^{-r_{n}} \right) - r_{n}. \nonumber
\end{eqnarray}
\noindent
Therefore
\begin{equation}
\frac{p-1}{n} 
\sum_{i \equiv b_{j} \bmod p} \nu_{p}(Q(i))  \geq 1 - p^{-r_{n}} -
\frac{(p-1)r_{n}}{n} 
\nonumber
\end{equation}
\noindent
and passing to the limit we conclude that
\begin{equation}
\liminf_{n \to \infty} \frac{p-1}{n} 
\sum_{i \equiv b_{j} \bmod p} \nu_{p}(Q(i))  \geq 1. 
\end{equation}
\noindent
Similarly, using the upper bound in (\ref{estimate1}) we obtain
\begin{equation}
\sum_{i \equiv b_{j} \bmod p} \nu_{p}(Q(i))  \leq r_{n} + \frac{n}{p-1},
\nonumber
\end{equation}
\noindent
and it follows  that
\begin{equation}
\limsup_{n \to \infty} \frac{p-1}{n} 
\sum_{i \equiv b_{j} \bmod p} \nu_{p}(Q(i))  \leq 1. 
\end{equation}
\noindent
Therefore, Theorem \ref{main} has been established.

\section{Examples} \label{sec-examples}
\setcounter{equation}{0}

In this section we present some examples illustrating Theorem \ref{main}.

\begin{Definition}
Given a polynomial $Q$ and a prime $p$, we
say that $a \in \mathbb{Z}/p \mathbb{Z}$ is a {\em Hensel zero} of $Q$ if 
$Q(a) \equiv 0 \bmod p$ and $Q'(a) \not \equiv 0 \bmod p$. The prime $p$ is 
called a {\em Hensel prime} for $Q$ if all the zeros of $Q$ in 
$\mathbb{Z}/p \mathbb{Z}$  are Hensel zeros.  We also require 
that $Q$ has at least
one zero in $\mathbb{Z}/p \mathbb{Z}$. 
The {\em asymptotic zero number} is defined (provided it exists) by the limit 
\begin{equation}
N_{p}(Q) := \lim\limits_{n \to \infty} \frac{(p-1) \nu_{p}(t_{n}) }{n}.
\end{equation}
\end{Definition}

Theorem \ref{main} is restated as follows:

\begin{Thm}
Let $p$ be a Hensel prime for
$Q$. Then $N_{p}(Q) = z_{p}$.
\end{Thm}

\noindent
{\bf Note}. The examples will show pairs $(Q,p)$ 
for which $N_{p}(Q) \not \in \mathbb{N}$. 
An appropriate interpretation of this number is lacking in these cases.

In the examples described below we present the 
{\em normalized error}
\begin{equation}
\text{err}_{p}(n;Q) := z_{p}n - (p-1)\nu_{p}(t_{n}(Q)) 
\label{norerror}
\end{equation}
\noindent
and the {\em relative error}:
\begin{equation}
\text{relerr}_{p}(n;Q):= \text{err}_{p}(n;Q) - \text{err}_{p}(n-1;Q).
\label{diferror}
\end{equation}
\noindent
Certain regular structure  of this function, as seen in Figure 
\ref{fig-1c}, will be analyzed in a future report.  \\

\noindent
{\bf Example 1}. 
Let $Q(x) = x^{5} + 2x^{3} +3$. Then $p=5$ is a Hensel prime for $Q$. 
Indeed, the only zeros of $Q$ in $\mathbb{Z}/5 \mathbb{Z}$ are $a=3$ and $a=4$
and $Q'(a) \not \equiv 0 \bmod 5$. Theorem \ref{main} gives 
\begin{equation}
\nu_{5}(t_{n}(Q)) \sim \frac{n}{2}. 
\end{equation}
\noindent
Figure \ref{fig-1a} shows the valuation $\nu_{p}(t_{n}(Q))$. Figure 
\ref{fig-1b} and \ref{fig-1c} depict patterns in the normal and relative 
error, respectively.

{{
\begin{figure}[ht]
\begin{center}
\includegraphics[width=3in]{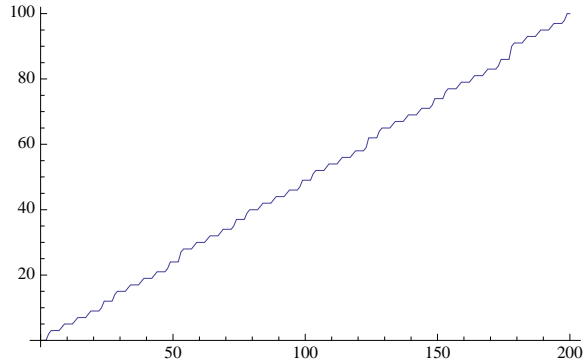}
\caption{The valuation $\nu_{5}(t_{n})$ for $Q(x)  = x^5 + 2x^{3} + 3$.}
\label{fig-1a}
\end{center}
\end{figure}
}}

{{
\begin{figure}[ht]
\begin{center}
\includegraphics[width=3in]{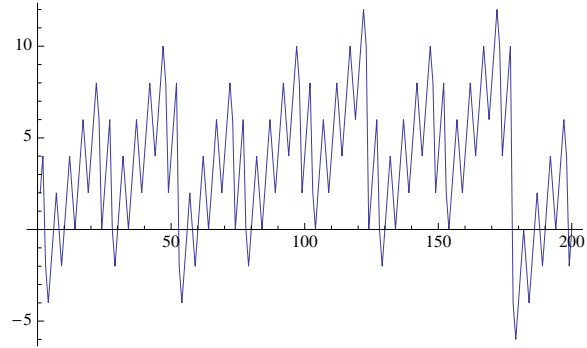}
\caption{The normalized error when $p=5$ and  $Q(x)  = x^5 + 2x^{3} + 3$.}
\label{fig-1b}
\end{center}
\end{figure}
}}

{{
\begin{figure}[ht]
\begin{center}
\includegraphics[width=3in]{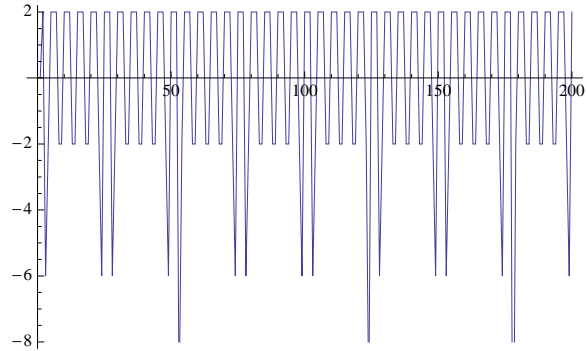}
\caption{The relative error when $p=5$ and $Q(x)  = x^5 + 2x^{3} + 3$.}
\label{fig-1c}
\end{center}
\end{figure}
}}

\medskip

\noindent
{\bf Example 2}. A direct calculation shows that, among the first $20000$
primes,  $p= 3, \, 11$ and $29$ are the 
only non-Hensel primes for $Q(x) = x^{5} + 2x^{3} + 3$. We now describe the 
asymptotic behavior of $\nu_{p}(t_{n}(Q))$ in each of these cases. 
The polynomial $Q$ factors as
\begin{equation}
x^5 + 2x^3 + 3 = (x+1)H(x) 
\end{equation}
\noindent
where 
\begin{equation}
H(x) = x^4-x^3+3x^2-3x+3
\end{equation}
\noindent
and the valuation splits  as
\begin{equation}
\nu_{p}(t_{n}(Q)) = \nu_{p}(t_{n}(x+1)) + \nu_{p}(t_{n}(H(x)).
\end{equation}
\noindent
Theorem \ref{main} gives $\nu_{p}(t_{n}(x+1)) \sim n/(p-1)$, so it remains to
evaluate  $\nu_{p}(t_{n}(H))$. \\

\noindent
{\bf The prime $p=3$.} 
In this case $0$ and $1$
are zeros of $H$ in $\mathbb{Z}/3 \mathbb{Z}$, and only $1$ is a Hensel zero. 
Observe that
\begin{equation}
\nu_{3}(t_{n}(H)) = \sum_{j \equiv 0 \bmod 3} \nu_{3}(H(j)) + 
\sum_{j \equiv 1 \bmod 3} \nu_{3}(H(j)). 
\label{firstsum}
\end{equation}
\noindent
Since $1$ is a Hensel zero, the argument in the proof of Theorem \ref{main} 
implies that
\begin{equation}
\sum_{j \equiv 1 \bmod 3} \nu_{3}(H(j)) \sim \frac{n}{2}.
\end{equation}
\noindent
To analyze the first sum in (\ref{firstsum}), note that
\begin{equation}
H(3k) = 81k^{4} - 27k^{3} + 27k^{2} - 9k + 3.
\end{equation}
\noindent
Thus, $\nu_{3}(H(3k)) = 1$ for $k \in \mathbb{N}$. We obtain that
\begin{equation}
\sum_{j \equiv 0 \bmod 3} \nu_{3}(H(j)) \sim \frac{n}{3},
\end{equation}
and then $\nu_{3}(t_{n}(H)) \sim \frac{5n}{6}$.  Therefore 
$\nu_{3}(t_{n}(Q)) \sim  \frac{4n}{3}$ and $N_{3}(Q) = \tfrac{8}{3}$. 

\medskip

\noindent
{\bf The prime $p=11$}. For this prime, although 
Theorem \ref{main} does not apply to $Q$
itself, it is applicable to both factors $x+1$ and $H(x)$. And,  we deduce
\begin{equation}
\nu_{11}(t_{n}(Q)) \sim \frac{n}{10} + \frac{2n}{10} = \frac{3n}{10}.
\end{equation}
\noindent
Therefore $N_{11}(Q) = 3$. 

\medskip

\noindent
{\bf The prime $p=29$}.  In order to find the 
asymptotic behavior of $\nu_{29}(t_{n}(H))$, observe that $14$ is the only zero 
of $H$ in $\mathbb{Z}/29 \mathbb{Z}$ and 
\begin{equation}
H(29k+14) = 36221+303601k+956217k^2 + 1341395k^{3} + 707281k^{4}.
\end{equation}
\noindent
The valuations of the coefficients in $H(29k+14)$ are $1, \, 2, \, 2, \, 3,$ and
$4$, respectively. Therefore $\nu_{29}(H(j)) = 1$ if $j \equiv 1 \bmod 29$ 
and $0$ otherwise. We conclude that 
\begin{equation}
\nu_{29}(t_{n}(H)) \sim \frac{n}{29}.
\end{equation}
\noindent
Therefore $\nu_{29}(t_{n}(Q)) \sim  \frac{57n}{812} $ and 
$N_{29}(Q) = \tfrac{57}{29}$. \\

\noindent
{\bf Example 3}. The polynomial 
\begin{equation}
Q(x) = x^{8} + x^{5} + x^{3} + 1 = (x^{3}+1)(x^{5}+1)
\end{equation}
\noindent
does not have a Hensel prime. This follows from 
\begin{equation}
\text{gcd} \left( Q(x), Q'(x) \right) = x+1,
\end{equation}
\noindent
so that, for any prime $p$, we have that 
$p-1$ is a zero of $Q$ in $\mathbb{Z}/p \mathbb{Z}$ and $Q'(p-1) = 0$. 
Naturally we have
\begin{equation}
\nu_{p}(t_{n}(Q)) = \nu_{p}(t_{n}(x^{3}+1)) + \nu_{p}(t_{n}(x^{5}+1)).
\label{example4}
\end{equation}

The asymptotic behavior of $\nu_{p}(t_{n}(Q))$ is discussed next.

\begin{Lem}
\label{main-1}
Let $p$ be an odd prime  and  $ x \neq 1$. Then 
\begin{equation}
\nu_{p}(x^{p}-1) = \begin{cases} 
                         0 & \quad \text{ if } x \not \equiv 1 \bmod p \\
                      1 + \nu_{p}(x-1) & \quad \text{ if } x \equiv 1 \bmod p. 
          \end{cases}
\end{equation}
\end{Lem}
\begin{proof}
The first part is clear from the congruence $x^{p} \equiv x \bmod p$. To verify
the second assertion, write $x = kp+1$ and observe that
\begin{equation}
\nu_{p}(x^{p}-1) = \nu_{p} \left( \sum_{r=1}^{p} \binom{p}{r} k^{r} p^{r}
\right). 
\end{equation}
\noindent
For $r > 1$, the $p$-adic valuation of each term in the sum is greater than 
$2 + \nu_{p}(k)$. When $r=1$, it is exactly $2+ \nu_{p}(k)$. Then, putting
$k = \frac{x-1}{p}$ verifies the assertion.
\end{proof}

\begin{Cor}
Let $p$ be an odd prime and $x \in \mathbb{Z}, \,  x \neq 1$. Define 
\begin{equation}
T_{p}(x) = x^{p-1} + x^{p-2} + \cdots + 1. 
\end{equation}
\noindent
Then
\begin{equation}
\nu_{p}(T_{p}(x)) = \begin{cases} 
                         0 & \quad \text{ if } x \not \equiv 1 \bmod p \\
                         1 & \quad \text{ if } x \equiv 1 \bmod p. 
          \end{cases}
\end{equation}
\end{Cor}

\begin{Cor}
Let $p$ be a prime and $x \in \mathbb{Z}, \, x \neq -1$. Then 
\begin{equation}
\nu_{p}(x^{p}+1) = \begin{cases} 
                         0 & \quad \text{ if } x \not \equiv -1 \bmod p \\
               1 + \nu_{p}(x+1) & \quad \text{ if } x \equiv -1 \bmod p. 
          \end{cases}
\end{equation}
\end{Cor}
\begin{proof}
Replace $x$ by $-x$ in Lemma \ref{main-1}.
\end{proof}

\begin{Cor}
Let $p$ be a prime and $x \in \mathbb{Z}, \,  x \neq -1$. Define 
\begin{equation}
S_{p}(x) = x^{p-1} - x^{p-2} + \cdots - x + 1. 
\end{equation}
\noindent
Then
\begin{equation}
\nu_{p}(S_{p}(x)) = \begin{cases} 
                         0 & \quad \text{ if } x \not \equiv -1 \bmod p \\
                         1 & \quad \text{ if } x \equiv -1 \bmod p. 
          \end{cases}
\end{equation}
\end{Cor}

The number of roots of $x^{q}+1 \equiv 0 \bmod p$, that is, 
$z_{p}(x^{q}+1)$ stated in the Lemma below appears at the end of Section 
8.1 of \cite{ireland1}. \\

\begin{Lem}
\label{irosen}
Let $p$ and $q$ be primes. The number of solutions of 
the congruence $x^{p}+1 \equiv 0 \bmod q$ is $\text{gcd}(p,q-1)$.
\end{Lem}

\begin{Cor}
Let $p$ be an odd prime. Then 
\begin{equation}
\nu_{p}(t_{n}(x^{p} \pm 1)) \sim \frac{(2p-1)n}{p(p-1)}.
\label{form1}
\end{equation}
\noindent
If $q$ is a prime, $q \neq p$, then 
\begin{equation}
\nu_{q}(t_{n}(x^{p} \pm 1)) \sim \frac{\text{gcd}(p,q-1) \, n}{q-1}.
\end{equation}
\end{Cor}
\begin{proof}
Theorem \ref{main} gives $\nu_{p}(t_{n}(x+1)) \sim \frac{n}{p-1}$. The 
expression for $\nu_{p}(S_{p}(x))$ yields $\nu_{p}(S_{p}(x)) \sim n/p$.
The asymptotic behavior of $\nu_{q}(t_{n}(x^{p} \pm 1))$ follow directly 
from Theorem \ref{main}. 
\end{proof}

We now complete the analysis of 
\begin{equation}
\nu_{p}(t_{n}(Q)) = \nu_{p}(t_{n}(x^{3}+1)) + \nu_{p}(t_{n}(x^{5}+1)).
\label{example4a}
\end{equation}
\noindent
If $p \neq 3$ is a prime, then 
\begin{equation}
\nu_{p} \left( t_{n}(x^{3}+1) \right) \sim \frac{z_{p}(x^{3}+1) \, n}{p-1}.
\end{equation}
\noindent
Similarly, for $p \neq 5$ prime, we have
\begin{equation}
\nu_{p} \left( t_{n}(x^{5}+1) \right) \sim \frac{z_{p}(x^{5}+1) \, n}{p-1}.
\label{form2}
\end{equation}
\noindent
Thus, (\ref{form1}) and (\ref{form2}) yield 
\begin{equation}
\nu_{3}(t_{n}(Q))  \sim   \nu_{3}(t_{n}(x^{3}+1)) + \nu_{3}(t_{n}(x^{5}+1) ) 
 = \frac{5n}{6} + \frac{n}{2} = \frac{4n}{3}.
\nonumber
\end{equation}
\noindent
Similarly, $\nu_{5}(t_{n}(Q)) \sim 7n/10$. \\

Now let $p \neq 3, \, 5$ be a prime. Theorem \ref{main} now applies directly 
to give
\begin{equation}
\nu_{p}(t_{n}(Q)) \sim \frac{ \left[ z_{p}(x^{3}+1) + z_{p}(x^{5}+1) \right] 
\, n }{p-1}. 
\end{equation}
\noindent
Lemma \ref{irosen} yields
\begin{equation}
\nu_{p}(t_{n}(Q)) \sim \frac{ \left[ \text{gcd}(3,p-1) + 
\text{gcd}(5,p-1)  \right] 
\, n }{p-1}. 
\end{equation}

The asymptotic zero number is given by
\begin{equation}
N_{p}((x^{3}+1)(x^{5}+1)) = \begin{cases}
        \frac{8}{3} & \text{ if } p = 3 \\
        \frac{14}{5} & \text{ if } p = 5 \\
        \text{gcd}(3,p-1) + \text{gcd}(5,p-1)  & \text{ if } p \neq 3, \, 5. \\
            \end{cases}
\end{equation}

\medskip

\noindent
{\bf Example 4}.  Let $p$ be an arbitrary prime and define
\begin{equation}
A_{p}(x) = (px+1)^{2} \left( (p+1)x+1 \right). 
\end{equation}
\noindent
A direct calculation shows that $p$ is the only Hensel prime for $A_{p}$. 
Therefore 
\begin{equation}
\nu_{p}(t_{n}(A_{p})) \sim \frac{n}{p-1}. 
\end{equation}
\noindent
To compute the asymptotics for a prime  $q \neq p$, let $Q_{1}(x) = px+1$
and $Q_{2}(x) = (p+1)x+1$, and observe that 
\begin{equation}
\nu_{q}(t_{n}(A_{p})) = 2 \nu_{q}(t_{n}(Q_{1})) + 
\nu_{q}(t_{n}(Q_{2})). 
\end{equation}
\noindent
Theorem \ref{main} applies to both $Q_{1}$ and $Q_{2}$. The case for $Q_{1}$ is 
immediate since $px+1 \equiv 0 \bmod q$ has a unique solution. To evaluate 
$\nu_{q}(t_{n}(Q_{2}))$ observe that the number of  solutions of
$(p+1)x+1 \equiv 0 \bmod q$ is $0$ or $1$, according to whether $q$ divides 
$p+1$ or not. Thus 
\begin{equation}
\nu_{q}(t_{n}(A_{p}))  \sim  \frac{(2 + \omega_{p,q})n}{q-1}
\end{equation}
\noindent
where 
\begin{equation}
\omega_{p,q} = \begin{cases}
                1  & \text{ if } q \text{ divides } p+1, \\
                0  & \text{ otherwise.} 
             \end{cases}
\nonumber
\end{equation}
\noindent
We conclude that
\begin{equation}
N_{q}(A_{p}) = \begin{cases} 
                1 & \quad \text{ if } p = q, \\
                2 + \omega_{p,q} & \quad \text{ if } p \neq q. 
        \end{cases}
\end{equation}

\medskip

\section{Future directions} \label{sec-future}
\setcounter{equation}{0}

In this section we outline certain generalizations of the main result 
of the paper. 

A natural extension of Theorem \ref{main}
deals with the situation in which there is an element 
$b \in \mathbb{Z}/p \mathbb{Z}$  such that 
\begin{equation}
Q(b) \equiv  Q'(b) \equiv \cdots \equiv Q^{(k-1)}(b) \equiv 0 \bmod p.
\end{equation}
\noindent
The  question of how the multiplicities of the roots enter in the 
asymptotic behavior of $\nu_{p}(t_{n}(Q))$ appears to be a salient quest, and 
this will be addressed elsewhere. 

Another interesting continuation of the ideas presented in this 
paper would be  the 
study of $p$-adic valuation of sequences satisfying second order 
recurrences
\begin{equation}
t_{n} = Q_{1}(n) t_{n-1} + Q_{2}(n) t_{n-2},
\end{equation}
\noindent
with polynomials $Q_{1}$ and $Q_{2}$. This problem includes, classically, the
case of Fibonacci and Stirling numbers.

\bigskip

\no
{\bf Acknowledgments}. 
The work of the third author was partially funded by
$\text{NSF-DMS } 0409968$. The second author was partially supported as a 
graduate student by the same grant. \\

\bigskip


\end{document}